\documentclass[
submission
]{dmtcs-episciences}

\usepackage[utf8]{inputenc}
\usepackage{subfigure}
\usepackage{mathtools}
\usepackage{amsmath}
\usepackage[margin=1in]{geometry}
\usepackage{amsthm}
\usepackage{tikz}
\usepackage{amssymb}
\usepackage{url}
\usepackage{tikz-qtree}
\usetikzlibrary{positioning}
\newtheorem{prop}{Proposition}[section]
\newtheorem{lemma}[prop]{Lemma}
\newtheorem{thm}[prop]{Theorem}
\newtheorem{defn}[prop]{Definition}
\newtheorem{cor}[prop]{Corollary}
\usepackage[round]{natbib}

\author{Yonah Biers-Ariel\affiliationmark{1}\thanks{Research Supported by NSF Grant 1263009}
  \and Anant Godbole\affiliationmark{2}\thanks{Research Supported by NSF Grant 1263009}
  \and Elizabeth Kelley\affiliationmark{3}\thanks{Research Supported by NSF Grant 1263009}}
\title[Expected Number of Subsequences in Random Binary Strings]{Expected Number of Distinct Subsequences in Randomly Generated Binary Strings}
\affiliation{
 Rutgers University, USA\\
East Tennessee State University, USA\\
University of Minnesota, USA}
\received{2017-04-28}
\revised{2018-03-16}
\accepted{2018-06-01}
\begin{document}
%\publicationdetails{VOL}{2015}{ISS}{NUM}{SUBM}
\publicationdetails{19}{2018}{2}{10}{3287}
\maketitle
\begin{abstract}
When considering binary strings, it's natural to wonder how many distinct subsequences might exist in a given string. Given that there is an existing algorithm which provides a straightforward way to compute the number of distinct subsequences in a fixed string, we might next be interested in the expected number of distinct subsequences in random strings. This expected value is already known for random binary strings where each letter in the string is, independently, equally likely to be a 1 or a 0. We generalize this result to random strings where the letter 1 appears independently with probability $\alpha \in [0,1]$.  Also, we make some progress in the case of random strings from an arbitrary alphabet as well as when the string is generated by a two-state Markov chain.
\end{abstract}

%DMTCS is an open access scientific is implemented by the
%\emph{episcience} platform, see \cite{berthaud:hal-01002815} for an
%overview of the strategy. It combines high scientific and editorial
%quality with an open access policy. It is priceless, neither authors
%nor readers pay money for the access. Access is granted by giving
%episcience an irrevocable license to publish the articles, the
%copyright remains with the authors. The platform itself is run by
%French government services that do their best to warrant continuous
%access and a high quality of service.
%
%This document describes the use of the \texttt{dmtcs-episcience.cls}
%document class. It should be used
%\begin{center}
%  \emph{\textbf{for all DMTCS publications}}.
%\end{center}
%
%If you are still preparing a document for our previous \texttt{OJS}
%platform, please add \texttt{ojs} to the classes options, see
%Section~\ref{sec:options}.

\section{Introduction}
This paper uses the definitions for string and subsequence provided in \cite{Flaxman}. A binary string of length $n$ is some $A=a_1a_2...a_n \in \{0,1\}^n$, and another string $B$ of length $m \le n$ is a subsequence of $A$ if there exist indices $i_1 < i_2 < ...< i_m$ such that 
$$B=a_{i_1}a_{i_2}...a_{i_m}$$
We use the notation $B \preceq A$ when $B$ is a subsequence of $A$.

Now, suppose $T_n$ is a {\it fixed} binary string of length $n$. Then, define $t_i$ to be the $i^{th}$ letter of $T_n$, $T_i$ to be the string formed by truncating $T_n$ after the $i\textsuperscript{th}$ letter, and $\phi(T_n)$ to be the number of distinct subsequences in $T_n$. If we let $S_n$ be a {\it random} binary string of length $n$, it was shown in \cite{Flaxman} that when Pr$[s_i=1]=.5$ (that is, when each independently generated letter in $S_n$ is equally likely to be a 0 or a 1), then $E[\phi(S_n)] \sim k(\frac{3}{2})^n$ for a constant $k$. Later, \cite{Collins} improved this result by finding that $E[\phi(S_n)]=2(\frac{3}{2})^n-1$ under the same conditions.

In Section 2, we generalize the second result and find a formula for the expected value of $\phi(S_n)$ when Pr$[s_i=1]=\alpha \in (0,1)$. Since the cases when Pr$[s_i=1]$ is 0 or 1 are trivial, we will then have $E[\phi(S_n)]$ when $\alpha \in [0,1]$. Our method for finding this formula is very different from that used by  \cite{Collins}. We will define a new property of a string---the number of new distinct subsequences--and then use these numbers as the entries in a binary tree. Our formula is then given as a weighted sum of the entries in this tree.  In Section 3 we produce recursions for the expected number of subsequences in two more complicated cases, namely when (i) the string of letters is independently generated in a non-uniform fashion from an arbitrary alphabet; and (ii) the binary string is Markov-dependent.  We also show, using subadditivity arguments, that the expected number of distinct subsequences in the first case above is asymptotic to $c^n$ for some $c$, in the sense that
\[(E[\phi(S_n)])^{1/n}\to c\enspace(n\to\infty).\]
Finally, in Section 4 we indicate some important directions for further investigation.

% You may scarsely use \clearpage to advance to a new page if this
% improves the readability of the document structure
%\clearpage
\section{Main Result}\label{main}
In this section, we let Pr$[s_i=1]=\alpha \in [0,1]$. We first consider the trivial cases, when $\alpha \in \{0,1\}$.

\begin{prop} If $\alpha \in \{0,1\}$, then $\phi(S_n)=n$ for any string $S_n$. \end{prop}
\begin{proof} If $\alpha \in \{0,1\}$, then either $S_n=11...1$ or $S_n=00...0$. In either case, there is exactly one distinct subsequence of each length between 1 and $n$.
\end{proof}

With those cases dispensed with, we will assume for the remainder of this section that $\alpha \in (0,1)$. Before continuing, we need to establish three new definitions.  We recall that $S_n$ is a random variable that equals a fixed string $T_n$ according to a specified probability distribution.

\begin{defn}The \emph{weight} of $T_n$, a fixed binary string of length $n$, is $\phi(T_n) \cdot \text{Pr}[S_n=T_n]$. That is, the weight of $T_n$ is the number of distinct subsequences in $T_n$ times the probability that a random length $n$ string is $T_n$.
\end{defn}

\begin{defn}A \emph{new} subsequence of $T_n$ is a subsequence contained in $T_n$ but not contained in $T_{n-1}$. We will let $\nu(T_n)$ be the number of distinct new subsequences in a string $T_n$.
\end{defn}

\begin{defn} The \emph{new weight} of $T_n$ is the product $\nu(T_n) \cdot \text{Pr}[S_n=T_n]$.
\end{defn}

It will be useful to be able to compute the number of new distinct subsequences in a string; to do so we modify a result from \cite{Elzinga}. 

\begin{lemma}[Elzinga, Rahmann, and Wang] \label{ERW} Given $T_n$, let $l$ be the greatest number less than $n$ such that $t_l=t_n$, and if no such number exists, let $l=0$. Then,
\\
$$\nu(T_n)=
\begin{cases}
n &\mbox{if } l=0;
\\
\displaystyle \sum_{i=l}^{n-1} \nu(T_i) &\mbox{if } l>0.
\end{cases}$$
\end{lemma}
\noindent\textit{Proof:} Fix $T_n$ and suppose $l=0$. Then, without loss of generality, assume that $T_n$ consists of $n-1$ 0s followed by a 1. The new subsequences contained in $T_n$ are exactly those which contain a 1. There are $n$ such subsequences: $1, 01, 001,..., \underbrace{00...0}_{n-1}1$, so $\nu(T_n)=n$.
\\
\indent Now suppose $l>0$ and $U_k$ is a new subsequence of length $k$ in $T_n$. We assume that the last letter in $U_k$ is $t_n$ because otherwise $U_k$ is clearly not new. If $U_{k-1} \preceq T_{l-1}$, then we could use $t_l$ to complete $U_k$, so $U_k \preceq T_l$ and $U_k$ would not be new in $T_n$. Conversely, if $U_{k-1} \not\preceq T_{l-1}$, then $U_k$ cannot be completed by $t_l$, nor can it be completed by any other $t_i$ before $t_n$, so $U_k$ is a new subsequence in $T_n$. Therefore, there is one distinct new subsequence in $T_n$ for every distinct new subsequence found in some $T_i$ with $l \le i \le n-1$. Summing up all of those distinct new subsequences gives the number of distinct new subsequences in $T_n$.
{\raggedright\qed}
\\
\\
Now, let $B$ be a binary tree whose entries are binary strings, and let $B_{n,m}$ be the $m\textsuperscript{th}$ entry in the $n\textsuperscript{th}$ row of $B$. The root of $B$ is the empty string, each left child is its parent with a 1 appended, and each right child is its parent with a 0 appended. If we call the first row ``row 0", then row $n$ of this tree contains all length $n$ binary strings. Rows 0-3 of this tree are shown below:

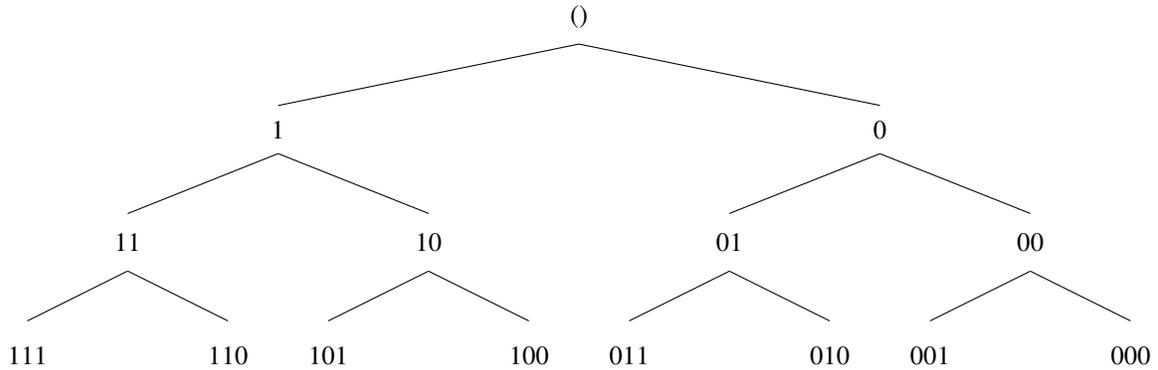
\begin{center}
\begin{tikzpicture}[node/.style={circle, draw=black!0, fill=green!0, thin, minimum size=2mm}, level/.style={sibling distance = 8cm/#1, level distance = 1.5cm},]
\node[node] () {()}
child
{
	node[node] {1}
	child
	{
		node[node]{11}
		child
		{
			node[node]{111}
		}
		child
		{
			node[node]{110}
		}
	}
	child
	{
		node[node] {10}
		child
		{
			node[node]{101}
		}
		child
		{
			node[node]{100}
		}
	}
}
child
{
	node[node] {0}
	child
	{
		node[node]{01}
		child
		{
			node[node]{011}
		}
		child
		{
			node[node]{010}
		}
	}
	child
	{
		node[node]{00}
		child
		{
			node[node]{001}
		}
		child
		{
			node[node]{000}
		}
	}
}; 
\end{tikzpicture}
\end{center}

Next, we form the binary tree $B'$ with $B'_{n,m}$ denoting the $m\textsuperscript{th}$ entry in the $n\textsuperscript{th}$ row of $B'$. Then, we define each $B'_{n,m}$ to be $\nu(B_{n,m})$ which we can calculate using Lemma 2.5. Finally, for each child $B'_{n,m}$ we assign the edge between it and its parent $B'_{n-1,\lceil\frac{m}{2}\rceil}$ a weight equal to Pr$[S_n=B_{n,m}|S_{n-1}=B_{n-1,\lceil\frac{m}{2}\rceil}]$. Thus we give each edge going to a left child the weight $\alpha$ and each edge going to a right child the weight $1-\alpha$. Rows 0-3 of $B'$ are shown below:
\\
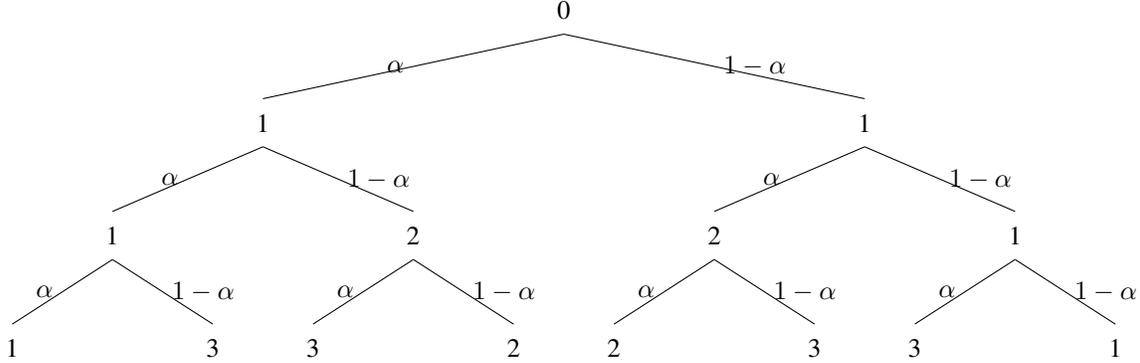
\begin{center}
\begin{tikzpicture}[node/.style={circle, draw=black!0, fill=green!0, thin, minimum size=2mm}, level/.style={sibling distance = 8cm/#1, level distance = 1.5cm},]
\node[node] (1) {0}
child
{
	node[node] {1}
	child
	{
		node[node]{1}
		child
		{
			node[node]{1}
			edge from parent node[left]{$\alpha$}
		}
		child
		{
			node[node]{3}
			edge from parent node[right]{$1-\alpha$}
		}
		edge from parent node[left]{$\alpha$}
	}
	child
	{
		node[node] {2}
		child
		{
			node[node]{3}
			edge from parent node[left]{$\alpha$}
		}
		child
		{
			node[node]{2}
			edge from parent node[right]{$1-\alpha$}
		}
		edge from parent node[right]{$1-\alpha$}
	}
	edge from parent node[left]{$\alpha$}
}
child
{
	node[node] {1}
	child
	{
		node[node]{2}
		child
		{
			node[node]{2}
			edge from parent node[left]{$\alpha$}
		}
		child
		{
			node[node]{3}
			edge from parent node[right]{$1-\alpha$}
		}
		edge from parent node[left]{$\alpha$}
	}
	child
	{
		node[node]{1}
		child
		{
			node[node]{3}
			edge from parent node[left]{$\alpha$}
		}
		child
		{
			node[node]{1}
			edge from parent node[right]{$1-\alpha$}
		}
		edge from parent node[right]{$1-\alpha$}
	}
	edge from parent node[right]{$1-\alpha$}
}; 
\end{tikzpicture}
\end{center}

It is clear that the value of the root should be 0 and the value of its two children should be 1. Moreover, we can apply Lemma 2.5 to find that each row begins and ends with a 1. To characterize the remaining entries of $B'$ we need two lemmas which together will give us that the portion of each row between the initial and final 1s consists of pairs of identical numbers. The numbers in every other pair are the sums of the parents of those elements, and the numbers in the remaining pairs have the same value as their parents. For instance, the first pair of the third row is a pair of 3s and parents of the elements of this pair are a 1 and a 2 which sum to 3. Then, the second pair of the third row is a pair of 2s, and the parents of the elements of this pair are also 2s. 

\begin{lemma}\label{joint} Suppose $n \ge 2$ and $m \equiv 2$ (mod 4). Then, $B'_{n,m}=B'_{n,m+1}=B'_{n-1,\frac{m}{2}}+B'_{n-1,\frac{m}{2}+1}$. \end{lemma}
\begin{proof}Given these restrictions on $n$ and $m$, both $B_{n,m}$ and $B_{n,m+1}$ are grandchildren of the same string $T_{n-2}$. Then, $B_{n,m}=T_{n-2}10$ and $B_{n,m+1}=T_{n-2}01$. First, consider the case when $T_{n-2}$ consists of only 0s or only 1s, and without loss of generality, assume it consists of only 0s. Using Lemma 2.5 we get the following equalities:
\begin{align*} \nu(T_{n-2}10)&=\nu(T_{n-2})+\nu(T_{n-2}1)=1+(n-1)=n; \\
\nu(T_{n-2}01)&=n; \\
\nu(T_{n-2}0)+\nu(T_{n-2}1)&=1+(n-1)=n,
\end{align*}
so the lemma is proved in this case. Otherwise, there exist $k,l >0$ such that $k$ is the greatest integer with $t_k=0$ and $l$ is the greatest integer with $t_l=1$. Then, the following hold, again by Lemma 2.5:
\begin{align*} \nu(T_{n-2}0)&=\displaystyle \sum_{i=k}^{n-2} \nu(T_{i}); \\
\nu(T_{n-2}1)&=\displaystyle \sum_{i=l}^{n-2} \nu(T_{i}); \\
\nu(T_{n-2}01)&=\displaystyle \sum_{i=l}^{n-2} \nu(T_{i})+\nu(T_{n-2}0)=\nu(T_{n-2}1)+\nu(T_{n-2}0); \\
\nu(T_{n-2}10)&=\displaystyle \sum_{i=k}^{n-2} \nu(T_{i})+\nu(T_{n-2}1)=\nu(T_{n-2}0)+\nu(T_{n-2}1).
\end{align*}
Since $T_{n-2}1=B_{n-1,\frac{m}{2}}$ and $T_{n-2}0=B_{n-1,\frac{m}{2}+1}$, this completes the proof.
\end{proof}
\begin{lemma}\label{single} Suppose $n \ge 2$, $m \equiv 0$ (mod 4), and $m \neq 2^n$. Then $B'_{n,m}=B'_{n,m+1}=B'_{n-1,\frac{m}{2}}=B'_{n-1,\frac{m}{2}+1}$ \end{lemma}
\begin{proof}
We will proceed by induction on $n$. As a base case, note that when $n=2$, the hypotheses of Lemma 2.7 are never satisfied, so it is true.

Now, suppose Lemma 2.7 holds for $n<p$, and consider $B'_{p,m}$ where $m$ satisfies the hypotheses of the lemma. Since $m \equiv 0$ (mod 4), it follows that $B_{p,m}=B_{p-1,\frac{m}{2}}0$, and $B_{p-1,\frac{m}{2}}$ also ends in 0. Then, $m+1 \equiv 1$ (mod 4), so $B_{p,m+1}=B_{p-1,\frac{m}{2}+1}1$, and $B_{p-1,\frac{m}{2}+1}$ also ends in 1. Therefore, by Lemma 2.5, $B'_{p,m}=B'_{p-1,\frac{m}{2}}$ and $B'_{p,m+1}=B'_{p-1,\frac{m}{2}+1}$. Now, $B'_{p-1,\frac{m}{2}}$ and $B'_{p-1,\frac{m}{2}+1}$ are consecutive elements and $\frac{m}{2}$ is even. If $\frac{m}{2} \equiv 2$ (mod 4), then $B'_{p-1,\frac{m}{2}}=B'_{p-1,\frac{m}{2}+1}$ by Lemma 2.6, and if $\frac{m}{2} \equiv 0$ (mod 4), then $B'_{p-1,\frac{m}{2}}=B'_{p-1,\frac{m}{2}+1}$ by the induction hypothesis. In either case, the proof is finished by induction.
\end{proof}

Now, the original question about distinct subsequences can reinterpreted as a question about the tree $B'$. In particular, if we find the path from each node $B'_{n,m}$ to the root and call the product of the weights of all the edges on that path $p_{n,m}$, we find that $p_{n,m}=\text{Pr}[S_n=B_{n,m}]$, and that
\begin{equation} \label{crux}
E[\phi(S_n)]= E\bigg[\displaystyle \sum_{i=1}^n \nu(S_i)\bigg]= \displaystyle \sum_{i=1}^n E[\nu(S_i)] =  \displaystyle \sum_{i=1}^n \displaystyle \sum_{j=1}^{2^i} B'_{i,j}\cdot p_{i,j}
\end{equation}
The first equality comes from the fact that each subsequence in $S_n$ is new exactly once, the second equality holds by linearity of expectation, and the third equality rewrites $E[\nu(S_i)]$ using the definition of expectation.

Now, recall the definition of new weight and note that the combined new weight of all the strings in row $i$ of $B$ is given by $\displaystyle \sum_{j=1}^{2^i} B'_{i,j}\cdot p_{i,j}$. With $i$ fixed, we define the quantity $a_i$ to be the total new weight of left children, i.e., strings ending in 1, in row $i$ of $B$ and $b_i$ to be the total new weight of right children, i.e., strings ending in 0, in row $i$ of $B$. We now find two simultaneous recurrence relations that describe $a_i$ and $b_i$.

\begin{lemma} The sequences $\{a_i\}$ and $\{b_i\}$ satisfy the following recurrence relations:  $$a_i=a_{i-1}+\alpha b_{i-1}; \quad b_i=b_{i-1}+(1-\alpha)a_{i-1}.$$\end{lemma}

\begin{proof}
Consider the left children in row $i$ of $B'$. First consider the left children whose parents are also left children. Each of these has the same value as its parent by Lemma 2.7, so the combined new weight of all of them is $\alpha \cdot a_{i-1}$. Now consider the left children whose parents are right children. By Lemma 2.6, each of these has a value which is the sum of the values of its parent, a right child, and this parent's sibling, a left child. In this way, each right child in row $i-1$ contributes its new weight times $\alpha$ to its left child in row $i-1$. Meanwhile, each left child in row $i-1$ contributes to the left child of its sibling its new weight times $1-\alpha$ since its path to the root contains one fewer edges labelled $1-\alpha$ than the path from the row $i$ left child to the root, but both paths contain the same number of edges labelled $\alpha$. Thus we get:
$$a_i=\alpha  a_{i-1} +\alpha  b_{i-1} + (1-\alpha)a_{i-1}=a_{i-1}+\alpha  b_{i-1}$$
The second recurrence relation follows by the same argument. The right children of row $i-1$ contribute their new weight times $1-\alpha$ to their own right children, and also contribute their new weight times $\alpha$ to the right children of their siblings. Meanwhile, the left children of row $i-1$ contribute their new weight times $1-\alpha$ to their right children, and so we get:
$$b_i=b_{i-1}+(1-\alpha)a_{i-1},$$
which completes the proof.\end{proof}

Now, we solve for $a_i$ as follows:
\begin{align*} \alpha b_{i-1}&=a_i-a_{i-1} \\
\alpha(1-\alpha)a_{i-1}&=\alpha b_i-\alpha b_{i-1} \\
\alpha(1-\alpha)a_{i-1}&=a_{i+1}-a_i -(a_i-a_{i-1}) \\
a_{i+1}&=2a_i-(1-\alpha(1-\alpha))a_{i-1} \end{align*}
The quadratic $x^2=2x-(1-\alpha(1-\alpha))$ has two real solutions: $\Big(1+\sqrt{\alpha(1-\alpha)}\Big)$ and $\Big(1-\sqrt{\alpha(1-\alpha)}\Big)$, so $a_i=c_1\Big(1+\sqrt{\alpha(1-\alpha)}\Big)^i+c_2\Big(1-\sqrt{\alpha(1-\alpha)}\Big)^i$ where $c_1$ and $c_2$ are constants. Inspecting the tree $B'$ gives that $a_1=\alpha$ and $a_2=2\alpha-\alpha^2$, and it is straightforward to verify that the following is an explicit formula of the correct form satisfying the initial conditions:
$$a_i=\frac{1}{2}\Big(\Big(\alpha-\sqrt{\alpha(1-\alpha)}\Big)\Big(1-\sqrt{\alpha(1-\alpha)}\Big)^{i-1}+\Big(\alpha+\sqrt{\alpha(1-\alpha)}\Big)\Big(1+\sqrt{\alpha(1-\alpha)}\Big)^{i-1}\Big)$$
To find an explicit formula for $b_i$, we simply note that if we substitute $a_i$ for $b_i$, $b_i$ for $a_i$, and $\alpha$ for $1-\alpha$, we obtain the recurrence that we just solved. Thus, making the reverse substitutions, we get that 
$$b_i=\frac{1}{2}\Big(\Big(1-\alpha-\sqrt{\alpha(1-\alpha)}\Big)\Big(1-\sqrt{\alpha(1-\alpha)}\Big)^{i-1}+\Big(1-\alpha+\sqrt{\alpha(1-\alpha)}\Big)\Big(1+\sqrt{\alpha(1-\alpha)}\Big)^{i-1}\Big)$$
and combining these two expressions gives the total new weight in row $i$ as
$$a_i+b_i=\frac{1}{2}\Big(\Big(1-2\sqrt{\alpha(1-\alpha)}\Big)\Big(1-\sqrt{\alpha(1-\alpha)}\Big)^{i-1}+\Big(1+2\sqrt{\alpha(1-\alpha)}\Big)\Big(1+\sqrt{\alpha(1-\alpha)}\Big)^{i-1}\Big)$$
As suggested by (\ref{crux}), the final step will be to find the sum $\displaystyle \sum_{i=0}^n (a_i +b_i)$. This follows from the geometric sum formula, and we get

\resizebox{\hsize}{!}{$\displaystyle \sum_{i=0}^n (a_i +b_i)=\frac{\Big(\Big(1-2\sqrt{\alpha(1-\alpha)}\Big)\Big(1-\Big(1-\sqrt{\alpha(1-\alpha)}\Big)^n\Big)+\Big(1+2\sqrt{\alpha(1-\alpha)}\Big)\Big(\Big(1+\sqrt{\alpha(1-\alpha)}\Big)^n-1\Big)}{2{\sqrt{\alpha(1-\alpha)}}}$}

We have now derived the main theorem of this section.

\begin{thm} Suppose Pr$[s_i=1]=\alpha \in [0,1]$ for all $1 \le i \le n$. Then we have

\resizebox{\hsize}{!}{$\phi(S_n)=
\begin{cases}
n &\mbox{if } \alpha=0,1,
\\
\frac{\big(1-2\sqrt{\alpha(1-\alpha)}\big)\big(1-\big(1-\sqrt{\alpha(1-\alpha)}\big)^n\big)+\big(1+2\sqrt{\alpha(1-\alpha)}\big)\big(\big(1+\sqrt{\alpha(1-\alpha)}\big)^n-1\big)}{2{\sqrt{\alpha(1-\alpha)}}} &\mbox{if } \alpha \neq 0,1.
\end{cases}
$} \end{thm}
Since this formula is rather unwieldy for $\alpha \in (0,1)$, we also give the following asymptotic result.

\begin{cor} Suppose Pr$[s_i=1]=\alpha \in (0,1)$ for all $1 < i < n$. Then there exists a constant $k$ such that
$$\phi(S_n) \sim k \big(1+\sqrt{\alpha(1-\alpha)}\big)^n.$$\end{cor}
\begin{proof}
We take the limit of the quantity in Theorem 2.9. Since $\big(1-\sqrt{\alpha(1-\alpha)}\big)<1$, it follows that $\displaystyle \lim_{n \rightarrow \infty} \big(1-\sqrt{\alpha(1-\alpha)}\big)^n =0$. Therefore 
$$\displaystyle \lim_{n \rightarrow \infty} \frac{\phi(S_n)}{\big(1+\sqrt{\alpha(1-\alpha)}\big)^n}\cdot \frac{2\sqrt{\alpha(1-\alpha)}}{\big(1+2\sqrt{\alpha(1-\alpha)}\big)}=1,$$
as asserted.\end{proof}
\section{Variations: Larger Alphabets, Markov Chains, and Growth Rates}
In the previous section, we looked at strings on a binary alphabet generated by a random process in which the probability that any given element was 1 was fixed at $\alpha$. In this section, we generalize this in two ways. First, we consider strings on the alphabet $\{1,2,...,d\}=[d]$ where each letter is independently $j$ with probability $\alpha_j$ for all $j \in [d]$. After that, we return to binary strings, but this time they will be generated according to a two-state Markov chain; in particular, if a letter follows a 1, then it is 1 with probability $\alpha$, but if it follows a 0, then it is 1 with probability $\beta$. In both these cases, we will find recurrences for the expected new weight contributed by the $n\textsuperscript{th}$ letter, which will lead to explicit matrix equations for that expected new weight. Unfortunately, we will not be able to find a closed-form formula for the total expected number of subsequences like we did in Section \ref{main}.
\subsection{A Larger Alphabet}
In this section, we consider strings on the alphabet $[d]$. In Section \ref{main}, $T_n$ was a fixed length-$n$ string on the alphabet $\{0,1\}$; here we let $T_n$ be a fixed length-$n$ string on the alphabet $[d]$. Similarly, $S_n$ is now a random length-$n$ string on the alphabet $[d]$ where, independently, Pr$[s_i=j] = \alpha_j$ for all $i \in [n] ,j \in [d]$ (note that $\sum_{j=1}^d \alpha_j = 1$).

We begin by stating a generalization of Lemma \ref{ERW}. The first paragraph of the proof is the same as in the proof of the original lemma, but is included for completeness.
\begin{lemma}\label{genERW}
Given $T_n$, let $l$ be the greatest number less than $n$ such that $t_l=t_n$, and if no such number exists, let $l=0$. Then,
\\
$$\nu(T_n)=
\begin{cases}
\displaystyle\sum_{i=1}^{n-1} \nu(T_i) + 1 &\mbox{if } l=0,
\\
\displaystyle \sum_{i=l}^{n-1} \nu(T_i) &\mbox{if } l>0.
\end{cases}$$
\end{lemma}

\begin{proof}Suppose $l>0$ and $U_k$ is a subsequence of length $k$ in $T_n$. We assume that the last letter in $U_k$ is $t_n$ because otherwise $U_k$ is clearly not new. If $U_{k-1} \preceq T_{l-1}$, then we could use $t_l$ to complete $U_k$, so $U_k \preceq T_l$ and $U_k$ would not be new in $T_n$. Conversely, if $U_{k-1} \not\preceq T_{l-1}$, then $U_k$ cannot be completed by $t_l$, nor can it be completed by any other $t_i$ before $t_n$, so $U_k$ is a new subsequence in $T_n$. Therefore, there is one distinct new subsequence in $T_n$ for every distinct new subsequence found in some $T_i$ with $l \le i \le n-1$. Counting up all of those distinct new subsequences gives the number of distinct new subsequences in $T_n$.

Now suppose $l=0$ and $U_k$ is a subsequence of length $k$ in $T_n$. As before, we can assume that the last letter of $U_k$ is $t_n$, but this time every subsequence of length ending in $t_n$ is new in $T_n$. Therefore, there is nearly a bijection between the subsequence which appeared by $T_{n-1}$ and the ones which are new in $T_n$ given by mapping a subsequence appearing by $T_{n-1}$ to itself with $t_n$ appended. This is not exactly a bijection because the single-element sequence $t_n$ is a new subsequence in $T_n$ even though the empty set is not counted as a subsequence in $T_{n-1}$. Therefore, the number of new subsequences is the sum of the number of subsequences which have appeared previously, plus one.
\end{proof}

Just like in the previous section, we are interested in the expected new weight of $S_n$. As before, we find it useful to refer to a tree $B$, but in this case it is a $d$-ary tree where the root is labelled with the empty string, and the $j\textsuperscript{th}$ child of a node is labelled with the parent's label with $j$ appended. The following figure shows rows 0-2 and part of row 3 of $B$ when $d=3$.

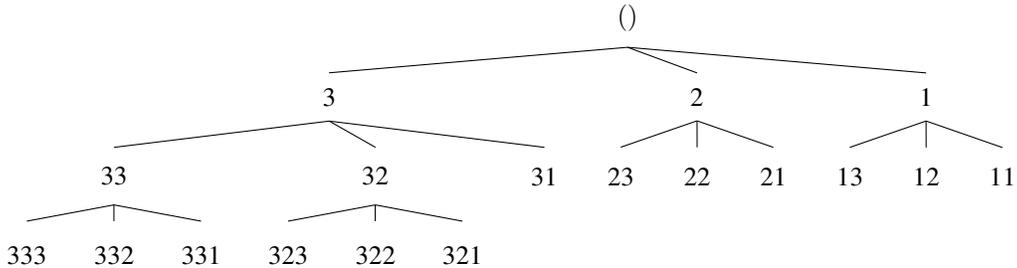
\begin{center}
\begin{tikzpicture}[every tree node/.style={draw=black!0,circle},sibling distance=.25cm]
\matrix{
\Tree
 [ .$()$
    [ .3 
    	[ .33 333 332 331 ]
	[ .32 323 322 321 ] 
	[ .31 ]
   ]
    [ .2 23 22 21 ]
    [ .1 13 12 11 ]
  ]
 
\\};
\end{tikzpicture}
\end{center}

Again like in the previous section, we will also use the tree $B'$ where each node $T_n$ of $B$ has been replaced by $\nu(T_n)$.

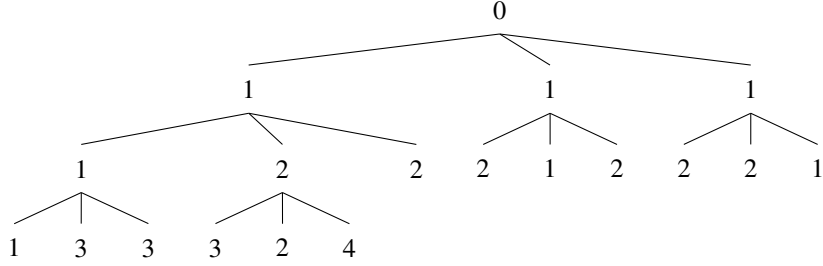
\begin{center}
\begin{tikzpicture}[every tree node/.style={draw=black!0,circle},sibling distance=.25cm]
\matrix{
\Tree
 [ .0
    [ .1
    	[ .1 1 3 3 ]
	[ .2 3 2 4 ] 
	[ .2 ]
   ]
    [ .1 2 1 2 ]
    [ .1 2 2 1 ]
  ]
 
\\};
\end{tikzpicture}
\end{center}

We now need to generalize Lemmas \ref{joint} and \ref{single}.

\begin{defn} A $(j,k)$-grandchild is a string whose final two letters are $j$ and $k$; equivalently, a $(j,k)$-grandchild is a string which is formed by appending $jk$ to its grandparent.
\end{defn}

\begin{lemma}\label{genjoint} Suppose $n \ge 2$ and $T_n$ is a $(j,k)$-grandchild with $j \neq k$ whose grandparent is $T_{n-2}$; then $\nu(T_n)=\nu(T_{n-2} j)+\nu(T_{n-2} k). $\end{lemma}

\begin{proof}As in Lemma \ref{genERW}, let $l_j$ and $l_k$ be the greatest numbers less than or equal to $n-2$ with $t_{l_j}=j$ and $t_{l_k}=k$ (or 0 if no such $t$ exist). Suppose that $l_k \neq 0$. By Lemma \ref{genERW}, we have that 
\[\nu(T_n)=\sum_{i=l_k}^{n-1} \nu(T_i) = \nu(T_{n-2}j) + \sum_{i=l_k}^{n-2} = \nu(T_{n-2}j) + \nu(T_{n-2}k).\] Otherwise, $l_k = 0$ and by Lemma 3.1, we have that 
\[\nu(T_n)=\sum_{i=1}^{n-1} \nu(T_i) + 1 = \nu(T_{n-2}j) + \sum_{i=1}^{n-2} \nu(T_i) + 1 = \nu(T_{n-2}j) + \nu(T_{n-2}k).\]
This completes the proof.\end{proof}

\begin{lemma}\label{gensingle} Suppose $n \ge 2$ and $T_n$ is a $(j,j)$-grandchild whose grandparent is $T_{n-2}$. Then, $\nu(T_n) = \nu(T_{n-2}j)$.\end{lemma}
\begin{proof}Since $T_n=T_{n-2}jj$ and $T_{n-1}=T_{n-2}j$, the statement follows immediately from Lemma \ref{genERW}.
\end{proof}

We want to be able to calculate the total new weight in each row of the tree $B'$. As in Section \ref{main} it will be convenient to break up that new weight by the final letter of the string it comes from. For each $n \in \mathbb{Z}^+$ and $j \in [d]$, let $a_{j,n}$ be the total new weight of the length-$n$ strings ending in $j$. Using Lemmas \ref{genjoint} and \ref{gensingle} we find that a typical element $T_{n-2}j$ contributes its weight times $\alpha_k$ to $T_{n-2}jk$ (for all $k \in [d]$) and also contributes its weight times $\alpha_k$ to $T_{n-2}kj$ for all $k \in [d] \setminus \{j\}$. Therefore, we obtain the recurrences
\[a_{j,n}=\sum_{k=1}^d \alpha_j a_{k,n-1} + \sum_{k \neq j} \alpha_k a_{j,n-1} = a_{j,n-1} + \sum_{k \neq j} \alpha_j a_{k,n-1}.\]

Using these $d$ recursions and the fact that the $a_{j,1}=\alpha_j$ for all $j$, we find that the vector $[a_{1,n},a_{2,n},...,a_{d,n}]^T$ is given by the matrix equation:
\[\begin{bmatrix}
a_{1,n}
\\
a_{2,n}
\\
a_{3,n}
\\
\vdots
\\
a_{d,n}
\end{bmatrix}
=
\begin{bmatrix}
1 & \alpha_1 & \alpha_1 & \dots & \alpha_1
\\
\alpha_2 & 1 & \alpha_2 & \dots & \alpha_2
\\
\alpha_3 & \alpha_3 & 1 & \dots & a_3
\\
 & & & \ddots
\\
\alpha_d & \alpha_d & \alpha_d & \dots & 1
\end{bmatrix}^{n-1}
\begin{bmatrix}
\alpha_1
\\
\alpha_2
\\
\alpha_3
\\
\vdots
\\
\alpha_d
\end{bmatrix},\]
and therefore the total new weight in row $n$ is 
\[\begin{bmatrix}
1 & 1 & 1 & \dots & 1
\end{bmatrix}
\begin{bmatrix}
1 & \alpha_1 & \alpha_1 & \dots & \alpha_1
\\
\alpha_2 & 1 & \alpha_2 & \dots & \alpha_2
\\
\alpha_3 & \alpha_3 & 1 & \dots & a_3
\\
 & & & \ddots
\\
\alpha_d & \alpha_d & \alpha_d & \dots & 1
\end{bmatrix}^{n-1}
\begin{bmatrix}
\alpha_1
\\
\alpha_2
\\
\alpha_3
\\
\vdots
\\
\alpha_d
\end{bmatrix}.\]
Therefore, we have a way to compute the expected value of $\phi(S_n)$ in this general alphabet case, as expressed in the following theorem.
\begin{thm} Let $S_n$ be a random length-$n$ string on the alphabet $[d]$ where Pr$[s_i=j]=\alpha_j$ for all $i,j$. Then,
\[E[\phi(S_n)]= \begin{bmatrix}
1 & 1 & 1 & \dots & 1
\end{bmatrix}
\left(\sum_{i=0}^{n-1}
\begin{bmatrix}
1 & \alpha_1 & \alpha_1 & \dots & \alpha_1
\\
\alpha_2 & 1 & \alpha_2 & \dots & \alpha_2
\\
\alpha_3 & \alpha_3 & 1 & \dots & a_3
\\
 & & & \ddots
\\
\alpha_d & \alpha_d & \alpha_d & \dots & 1
\end{bmatrix}^{i}
\right)
\begin{bmatrix}
\alpha_1
\\
\alpha_2
\\
\alpha_3
\\
\vdots
\\
\alpha_d
\end{bmatrix}.
\]
\end{thm}

\subsection{Strings from Markov Processes}
We now return to binary strings, but now the probability of seeing a particular letter will depend on the letter before (we assume that strings are generated from left to right). In the random string $S_n$, Pr$[s_i=1 | s_{i-1}=1] = \alpha$ and Pr$[s_i = 1 | s_{i-1}=0]= \beta$. Of course, we will need some other rule for Pr$[s_1=1]$; one logical choice is to take Pr$[s_1=1]=\gamma$ where $\gamma$ is the steady-state probability of a 1 occurring, which in this case gives $\gamma=\frac{\beta}{1+\beta-\alpha}$. Unlike in the previous subsection, we don't need any new lemmas to discuss this case. Instead, the key will be to categorize the new weight in each row by the last \emph{two} letters of the strings which contribute it, rather than just by the last single letter. 

To that end, let $a_n$ be the total new weight of $n$-letter strings ending in 11, let $b_n$ be the total new weight of $n$-letter strings ending in 10, let $c_n$ be the total new weight of $n$ letter strings ending in 01, and let $d_n$ be the total new weight of $n$-letter strings ending in 00. We apply Lemmas \ref{joint} and \ref{single} in much the same manner that we did in Section \ref{main} to find the recurrences:
\[a_{n+1}=\alpha(a_n+c_n);\]
\[b_{n+1}=(1-\alpha)(a_n+c_n)+\alpha b_n +\frac{\beta(1-\alpha)}{1-\beta} d_n;\]
\[c_{n+1}=\beta(b_n+d_n)+(1-\beta)c_n+\frac{(1-\alpha)\beta}{\alpha} a_n;\]
\[d_{n+1}=(1-\beta)(b_n+d_n).\]
Again, these recurrences lead to a matrix equations
\[\begin{bmatrix}
a_n
\\
b_n
\\
c_n
\\
d_n
\end{bmatrix}
=
\begin{bmatrix}
\alpha & 0 & \alpha & 0
\\
1-\alpha & \alpha & 1-\alpha & \frac{\beta(1-\alpha)}{1-\beta}
\\
\frac{(1-\alpha)\beta}{\alpha} & \beta & 1-\beta & \beta
\\
0 & 1-\beta & 0 & 1- \beta
\end{bmatrix}^{n-1}
\begin{bmatrix}
a_1
\\
b_1
\\
c_1
\\
d_1
\end{bmatrix},
\] and

\[E[\phi(S_n)]= \begin{bmatrix}
1 & 1 & 1 & \dots & 1
\end{bmatrix}
\left(\sum_{i=0}^{n-1}
\begin{bmatrix}
\alpha & 0 & \alpha & 0
\\
1-\alpha & \alpha & 1-\alpha & \frac{\beta(1-\alpha)}{1-\beta}
\\
\frac{(1-\alpha)\beta}{\alpha} & \beta & 1-\beta & \beta
\\
0 & 1-\beta & 0 & 1- \beta
\end{bmatrix}^{i-1}
\right)
\begin{bmatrix}
a_1
\\
b_1
\\
c_1
\\
d_1
\end{bmatrix}.
\]
While we could define Pr$[s_1=1]=\gamma$ (recalling that $\gamma = \frac{\beta}{1+\beta-\alpha}$) it will make our formula work out more nicely if we pretend that there exists a letter $s_0$ which does not add to any subsequences but which determines Pr$[s_1=1]$. If we take Pr$[s_0 = 1]=\gamma$, then we have $a_1=\gamma\alpha, b_1 = \gamma(1-\alpha), c_1 = (1-\gamma)\beta, d_1 = (1-\gamma)(1-\beta)$, and so the definition of the formula is complete.
\subsection{Exponential Growth via Fekete's Lemma}  In this subsection, we exhibit the fact that in the case of general alphabets, the expected number of distinct sequences is ``around" $c^n$, mirroring  the result in Corollary 2.10.  This fact is hardly surprising, but raises other questions, namely as to whether the ``true" numbers contain, additionally, polynomial factors as do several Stanley-Wilf limits in the theory of pattern avoidance (note that there are no polynomial factors in Corollary 2.10).

 Also, in general the existence of limits is not automatic, as seen by the following example:  Assume that $n$ balls are independently thrown into an infinite array of boxes so that box $j$ is hit with probability $1/2^j$ for $j=1,2,\ldots$.  Let $\pi_n$ be the probability that the largest occupied box has a single ball in it.  Then, as seen in \cite{athreya}, $\lim_{n\to\infty} \pi_n$ does not exist, and $\lim\sup_{n\to\infty}\pi_n$ and $\lim\inf_{n\to\infty}\pi_n$ differ in the fourth decimal place! Such behavior does not however occur in our context, as we prove next.
\begin{thm}  Let $s_1,s_2,\ldots $ be a sequence of independent and identically distributed random variables with $Pr(s_1=j)=\alpha_j, j=1,2,\ldots,d$, and $\sum_j\alpha_j=1$.  Set $\alpha=(\alpha_1,\ldots,\alpha_d)$.  Let $\phi(S_n)$ and $\phi(S_{n+1,n+m})$ be the number of distinct subsequences in $S_n=(s_1,\ldots,s_n)$ and $(s_{n+1},\ldots,s_{n+m})$.  Let $\psi(n)=E(\phi(S_n))$.  Then there exists $c=c_{d,\alpha}\ge 1$ such that
\[\psi(n)^{1/n}\to c; n\to\infty,\] where $c=1$ iff $d\ge 1$ and $\max_j \alpha_j=1$.
\end{thm}
\begin{proof} As in Arratia's paper \cite{arratia} on the existence of Stanley-Wilf limits, we employ subadditivity arguments and Fekete's lemma.  Assume without loss of generality that $m\le n$.  Let $\eta$ be a distinct subsequence of $S_{m+n}$, and consider the first lexicographic occurrence of $\eta$, say $\eta_f$.  Then $\eta_f\vert_{S_n}=\eta_{f_1}$ and $\eta_f\vert_{(s_{n+1},\ldots,s_{n+m})}=\eta_{f_2}$ are a pair of subsequences of $(s_1,\ldots,s_n)$ and $(s_{n+1},\ldots,s_{n+m})$.  Moreover, the map
\[\eta_f\longrightarrow(\eta_{f_1},\eta_{f_2})\] is one-to-one (note that one of the components $\eta_{f_1},\eta_{f_2}$ may be empty).  Thus
\[S_{n+m}\le S_nS_{n+1,n+m},\]
and thus, by independence we get
\[\psi(n+m)\le\psi(n)\psi(n+1,n+m).\]Since $m\le n$, we conclude that
\[\psi(n+m)\le\psi(n)\psi(m).\]
In other words, $\xi(n)=\log \psi(n)$ satisfies the subadditivity condition
\[\xi(n+m)\le \xi(n)+\xi(m),\] and Fekete's lemma yields the conclusion that
\[\frac{\xi(n)}{n}\to \ell=\inf_{n\ge1}\frac{\xi(n)}{n}.\]
Clearly $\ell\in[0,\log d]$.
We thus get
\[\frac{\log \psi(n)}{n}=\log [\psi(n)]^{1/n}\to \ell,\]
and so
\[\psi(n)^{1/n}\to e^\ell:=c,\]
where $c\in[1,d]$.  Clearly $c=1$ for any $d$ if $\alpha_1=1$.  We need to show that this is the only case when this occurs.  By Theorem 2.9, we know that $c>1$ if $d=2$ and $\max(\alpha_1,\alpha_2)\ne1$.  Using a monotonicity argument (for larger size alphabets, we replace all the letters $2,3,\ldots,d$ by 2), it is easy to see that $c>1$ if $d\ge3$ and $\max_j(\alpha_j)<1$.  This concludes the proof.
\end{proof}
\section{Discussion and Open Questions}  In this section we list and briefly discuss some questions that seem to be quite non-trivial.

(a) One of the central questions in the Permutation Patterns community is that of packing patterns and words in larger ensembles; see, e.g., \cite{burstein}.  In a similar vein, we have the question of superpatterns, i.e., strings that contain all the patterns or words of a smaller size; see, e.g., \cite{yonah}.  A distinguished question in this area is the one posed by Alon, who conjectured (see \cite{yonah}) that a random permutation on $[n]=\left[\frac{k^2}{4}(1+o(1))\right]$ contains all the permutations of length $k$ with probability asymptotic to 1 as $n\to\infty$.  In the present context, a similar question might be:  What is the largest $k$ so that each element of $\{0,1\}^k$ appears as a subsequence of a binary random string with high probability?

(b) The basic question studied in this paper appears to not have been considered in the context of permutations; i.e., one might ask: What is the expected number of distinct patterns present in a random permutation on $[n]$?  

(c) Computation of the rates of exponential growth in Theorem 3.6 would, of course, be of interest as would be, alternatively, a solution of the recurrences in Sections 3.1 and 3.2.  Also, estimation of the width of the intervals of concentration of the number of distinct subsequences, around their expected values, would add significantly to our understanding of the situation.  

(d) An intriguing question (which leads to a wide area for further investigation) is the following.  In the baseline case of binary equiprobable letter generation, we have that 
$E(\phi(S_n))\sim2(1.5)^n$, which implies that the {\it average} number of occurrence of a subsequence is $\frac{1}{2}{2^n}/{(1.5)^n}=\frac{1}{2}(4/3)^n$.  Now a subsequence such as $1$ occurs ``just" around $n/2$ times, and the sequence $11\ldots1$ with $n/2$ ones occurs an average of ${{n}\choose{n/2}}\cdot\frac{1}{2^{n/2}}$ times, which simplifies, via Stirling's formula, to around $\sqrt{2}^n$, ignoring constants and polynomial factors.  The same is true of any sequence of length $n/2$;  it is, on average, over-represented.  We might ask, however, what length sequences occur more-or-less an average number $(1.33)^n$ of times.  We can parametrize by setting $k=xn$ and equating the expected number of occurrences of a $k$-long sequence to $(1.33)^n$.  We seek, in other words, the solution to the equation
\[{{n}\choose{xn}}\frac{1}{2^{xn}}=(1.33)^n.\] 
Ignoring non-exponential terms and employing Stirling's approximation, the above reduces to
\[2^xx^x(1-x)^{1-x}=0.75,\] which, via Wolfram Alpha, yields the solutions $x=.123\ldots$ and $x=.570\ldots$!   In a similar fashion we see that the expected number of occurrences of a sequence of length $(0.7729\ldots)n$ or longer is smaller than one.  Does this suggest that the solution to the Alon-like question stated in (a) above might be $k=(0.7729\ldots)(1+o(1))n$?

\nocite{*}

%\bibliographystyle{abbrvnat}
% use the following instead if you encounter problems 
\bibliographystyle{alpha}
\bibliography{sample-dmtcs}
\label{sec:biblio}

\end{document}